\documentclass[a4paper,11pt]{article}
\usepackage{amsfonts}
\usepackage{amssymb}
\usepackage{graphicx}
\usepackage[font=small,labelfont=bf]{caption}
\usepackage{amsmath}
\usepackage{pifont}
\usepackage{verbatim}
\usepackage{color}
\usepackage{hyperref}
\usepackage{scrextend}
\usepackage{stmaryrd}

\newtheorem{theorem}{Theorem}

\newtheorem{corollary}[theorem]{Corollary}

\newtheorem{lemma}[theorem]{Lemma}

\newtheorem{proposition}[theorem]{Proposition}

\newenvironment{proof}[1][\textbf{Proof}]{\noindent\textbf{#1.} }{\mbox{ } \hfill \rule{0.5em}{0.5em}\medskip}

\makeatletter
\renewcommand*\env@cases[1][1.2]{
\let\@ifnextchar\new@ifnextchar
\left\lbrace
\def\arraystretch{#1}
\array{@{}l@{\quad}l@{}}
}
\makeatother

\begin{document}

\title{Note on a sign-dependent regularity for the polyharmonic Dirichlet
problem}
\author{Inka Schnieders\thanks{
Department Mathematik/Informatik, Universit\"{a}t zu K\"{o}ln, Weyertal
86-90, 50931 K\"{o}ln, Germany, ischnied@math.uni-koeln.de,
gsweers@math.uni-koeln.de}~ \& Guido Sweers$^{\ast }$}
\maketitle

\begin{abstract}
A priori estimates for semilinear higher order elliptic equations usually
have to deal with the absence of a maximum principle. This note presents
some regularity estimates for the polyharmonic Dirichlet problem that will
make a distinction between the influence on the solution of the positive and
the negative part of the right-hand side.
\end{abstract}

\smallskip

\noindent{\footnotesize \textbf{Mathematics Subject Classification (MSC)
2020:} primary 35J30, secondary 35B45, 35J86}

\section{Introduction and main result}

Let $\Omega \subset \mathbb{R}^{n}$ be a bounded domain with $\partial
\Omega \in C^{2m,\gamma }$ with $m\in \mathbb{N}^{+}$ and $\gamma \in \left(
0,1\right) $, and consider the Dirichlet problem for the poly-laplace
operator:%
\begin{equation}
\left\{
\begin{array}{cc}
\left( -\Delta \right) ^{m}u=f & \text{in }\Omega , \\
u=\frac{\partial }{\partial n}u=\dots =\left( \frac{\partial }{\partial n}%
\right) ^{m-1}u=0 & \text{on }\partial \Omega .%
\end{array}%
\right.  \label{poly}
\end{equation}%
Suppose that%
\begin{equation}
f^{+}:=\max \left( 0,f\right) \text{ \ and \ }f^{-}:=\max \left( 0,-f\right)
\label{pm}
\end{equation}
is such that $f^{+}\in L^{p_{+}}(\Omega )$ and $f^{-}\in L^{p-}(\Omega )$
with $p_{+},p_{-}\in \left( 1,\infty \right) $. For the second order case,
that is $m=1$, one may use the maximum principle and solve%
\begin{equation*}
\left\{
\begin{array}{c}
-\Delta u^{\oplus }=f^{+} \\
u^{\oplus }=0%
\end{array}%
\ \text{and\ }%
\begin{array}{cc}
-\Delta u^{\ominus }=f^{-} & \text{in }\Omega , \\
u^{\ominus }=0 & \text{on }\partial \Omega ,%
\end{array}%
\right.
\end{equation*}%
separately to find $u=u^{\oplus }-u^{\ominus }$\ for%
\begin{equation*}
0\leq u^{\oplus }\in W^{2,p_{+}}(\Omega )\cap W_{0}^{1,p_{+}}(\Omega )\text{
\ and \ }0\leq u^{\ominus }\in W^{2,p_{-}}(\Omega )\cap
W_{0}^{1,p_{-}}(\Omega ),
\end{equation*}%
with the usual regularity estimates (\cite{ADN1}):%
\begin{equation}
\left\Vert u^{\oplus }\right\Vert _{W^{2m,p_{+}}\left( \Omega \right) }\leq
c_{p_{+}}\left\Vert f^{+}\right\Vert _{L^{p_{+}}(\Omega )}\text{ and }%
\left\Vert u^{\ominus }\right\Vert _{W^{2m,p_{-}}\left( \Omega \right) }\leq
c_{p_{+}}\left\Vert f^{-}\right\Vert _{L^{p_{-}}(\Omega )}.  \label{mis2}
\end{equation}%
The constants will depend on $\Omega $, but that dependence we will suppress
in our notation.\medskip

Whenever $m\geq 2$ there is no maximum principle or, unless we have a
special domain like a ball \cite{Boggio}, a positivity preserving property
in the sense that $f\geq 0$ in (\ref{poly}) would result in $u\geq 0$.
Nevertheless, it is possible to find a result quite similar to (\ref{mis2})
for the solution of (\ref{poly}). Such a separation of the regularity for
the positive and negative part is something we need for a higher order
semilinear problem that we consider in \cite{SS3}. Since we believe it has
some interest in itself, we present this sign-dependent regularity in this
separate note.\medskip

Our main result for (\ref{poly}) with $m\in \mathbb{N}^{++}:=\left\{
2,3,\dots \right\} $ is as follows:

\begin{theorem}
\label{plusminusregthm}Let $\Omega \subset \mathbb{R}^{n}$ be a bounded
domain with $\partial \Omega \in C^{2m,\gamma }$ and let $p_{\pm }\in \left(
1,\infty \right) $. Suppose that $f=f^{+}-f^{-}$ as in (\ref{pm}) with $%
f^{+}\in L^{p_{+}}( \Omega ) $ and $f^{-}\in L^{p_{-}}( \Omega) $. Then
there exist constants $c_{p_{+},m},c_{p_{-},m}>0$, independent of $%
f^{+},f^{-}$, such that the following holds. The unique solution $u$ of (\ref%
{poly}) can be written as $u=u^{\oplus }-u^{\ominus }$, with%
\begin{gather*}
0\leq u^{\oplus }\in W^{2m,p_{+}}(\Omega )\cap W_{0}^{m,p_{+}}(\Omega ), \\
0\leq u^{\ominus }\in W^{2m,p_{-}}(\Omega )\cap W_{0}^{m,p_{-}}(\Omega ),
\end{gather*}%
and
\begin{gather*}
\left\Vert u^{\oplus }\right\Vert _{W^{2m,p_{+}}\left( \Omega \right) }\leq
c_{p_{+},m}\left( \left\Vert f^{+}\right\Vert _{L^{p_{+}}(\Omega
)}+\left\Vert f^{-}\right\Vert _{L^{1}(\Omega )}\right) , \\
\left\Vert u^{\ominus }\right\Vert _{W^{2m,p_{-}}\left( \Omega \right) }\leq
c_{p_{-},m}\left( \left\Vert f^{-}\right\Vert _{L^{p_{-}}(\Omega
)}+\left\Vert f^{+}\right\Vert _{L^{1}(\Omega )}\right) .
\end{gather*}
\end{theorem}

Although we will construct $u^{\oplus }$, $u^{\ominus }$ in a way such that $%
u^{\oplus }$, $u^{\ominus }$ is unique, the statement in the theorem does
not give uniqueness of this decomposition $u^{\oplus },u^{\ominus }$. Since $%
f\in L^{p}(\Omega )$ with $p=\min \left\{ p_{-},p_{+}\right\} >1$ and $%
\partial \Omega \in C^{2m,\gamma }$, the solution is unique in $%
W^{2m,p}(\Omega )\cap W_{0}^{m,p}(\Omega )$.\medskip

Generically $u^{\oplus }\neq u^{+}$, but since $u^{+}=\left( u^{\oplus
}-u^{\ominus }\right) ^{+}\leq \left( u^{\oplus }\right) ^{+}=u^{\oplus }$,
we find that%
\begin{equation}
-u^{\ominus }\leq -u^{-}\leq 0\leq u^{+}\leq u^{\oplus }.  \label{mmpp}
\end{equation}%
With this estimate one also finds a signed Sobolev inequality. Setting
\begin{equation}
q_{n,m,p}:=\frac{np}{n-2mp}  \label{qnmp}
\end{equation}%
we may combine with the Sobolev imbedding theorem, see \cite[Theorem 4.12]%
{AF}, to obtain the following:

\begin{corollary}
\label{cor}Let $\Omega \subset \mathbb{R}^{n}$ be bounded with $\partial
\Omega \in C^{2m,\gamma }$ and let $p_{\pm }\in \left( 1,\infty \right) $.
Suppose that $f=f^{+}-f^{-}$ with $f^{+}\in L^{p_{+}}\left( \Omega \right) $
and $f^{-}\in L^{p_{-}}\left( \Omega \right) $. Let $u$ be the solution of (%
\ref{poly}) as in Theorem \ref{plusminusregthm}. Then the following holds:

\begin{enumerate}
\item If moreover $p_{+}\leq \frac{n}{2m}$ (so $n>2m$) and $q\in \left[
1,q_{n,m,p_{+}}\right] $ with $q<\infty $, then there is $%
c_{p_{+},q,m}^{\prime }>0$ such that%
\begin{equation*}
\left\Vert u^{+}\right\Vert _{L^{q}(\Omega )}\leq c_{p_{+},q,m}^{\prime
}\left( \left\Vert f^{+}\right\Vert _{L^{p_{+}}(\Omega )}+\left\Vert
f^{-}\right\Vert _{L^{1}(\Omega )}\right) .
\end{equation*}

\item If moreover $p_{+}>\frac{n}{2m}$, then there is $c_{p_{+},m}^{\prime
}>0$ such that%
\begin{equation*}
\sup u\leq c_{p_{+},m}^{\prime }\left( \left\Vert f^{+}\right\Vert
_{L^{p_{+}}(\Omega )}+\left\Vert f^{-}\right\Vert _{L^{1}(\Omega )}\right) .
\end{equation*}
\end{enumerate}
\end{corollary}

Similar results depending on $p_{-}$ hold for $u^{-}$ and $\sup \left(
-u\right) $.

\section{Relation to previous results}

Since the fundamental contributions by Agmon, Douglis and Nirenberg \cite%
{ADN1} it is known, assuming that $\Omega $ is bounded with a smooth enough
boundary, that for each $p\in \left( 1,\infty \right) $ and $f\in
L^{p}\left( \Omega \right) $ a solution of (\ref{poly}) satisfies $u\in
W^{2m,p}\left( \Omega \right) $. Whenever the solution is unique, and with
the $C^{2m,\gamma}$-boundary the solution for (\ref{poly}) is unique for any
$p\in \left( 1,\infty \right) $, there exist $C_{m,p}>0$, independent of $f$%
, such that%
\begin{equation}
\left\Vert u\right\Vert _{W^{2m,p}(\Omega )}\leq C_{m,p}\left\Vert
f\right\Vert _{L^{p}(\Omega )}.  \label{adnest}
\end{equation}%
Whenever $p\in \left( 1,\frac{n}{2m}\right) $, and such $p$ exist when $n>2m$%
, the Sobolev imbedding shows that for $q\leq q_{n,m,p}$, as in (\ref{qnmp}%
), a constant $C_{m,p,q}^{\prime }>0$ exists such that
\begin{equation*}
\left\Vert u\right\Vert _{L^{q}(\Omega )}\leq C_{m,p,q}^{\prime }\left\Vert
u\right\Vert _{W^{2m,p}(\Omega )}.
\end{equation*}%
Combining both estimates will lead to an estimate as in Corollary (\ref{cor}%
) but then without the sign. However, since $0\leq u^{+}\leq u^{\oplus }$
holds, one finds $\left\Vert u^{+}\right\Vert _{L^{q}(\Omega )}\leq
\left\Vert u^{\oplus }\right\Vert _{L^{q}(\Omega )}$ and one is left with
proving the result in Theorem \ref{plusminusregthm}. \medskip

For $\Omega =\mathbb{R}^{n}$ signed estimates as in the corollary will
follow directly from the Riesz potential $\mathcal{I}_{2m}$, \cite{Riesz},
for the Riesz potential solution of $\left( -\Delta \right) ^{m}u=f$, when $%
f $ goes to zero at $\infty $ in an appropriate sense. Indeed, see \cite[%
Chapter V]{Stein}, that solution is given by
\begin{equation}
u(x)=\left( \mathcal{I}_{2m}f\right) (x):=\tfrac{\Gamma \left( \frac{1}{2}%
n-m\right) }{\pi ^{\frac{1}{2}n}4^{m}\Gamma \left( m\right) }\int_{\mathbb{R}%
^{n}}\left\vert x-y\right\vert ^{2m-n}f(y)dy.  \label{rp}
\end{equation}%
Since the kernel in (\ref{rp}) is positive, it allows one to consider
separately the influence of $f^{+}\in L^{p_{+}}(\mathbb{R}^{n})$ and $%
f^{-}\in L^{p_{-}}(\mathbb{R}^{n})$ with $p_{+},p_{-}\in \left( 1,\frac{n}{2m%
}\right) $. Indeed, on $\mathbb{R}^{n}$ the function $u=u^{\oplus
}-u^{\ominus }$ with
\begin{equation*}
u^{\oplus }(x):=\left( \mathcal{I}_{2m}f^{+}\right) (x)\text{ and }%
u^{\ominus }\left( x\right) :=\left( \mathcal{I}_{2m}f^{+}\right) (x)
\end{equation*}%
is such that $u^{\oplus }\in L^{q_{+}}(\mathbb{R}^{n})$ and $u^{\ominus }\in
L^{q_{-}}(\mathbb{R}^{n})$ with $q_{\pm }:=q_{n,m,p_{\pm }}$. \medskip

On a bounded domain with homogeneous Dirichlet boundary conditions the
crucial ingredient that allows us to consider $f^{+}$ and $f^{-}$
separately, comes from \cite{GRobS}. There one finds that the Green function
$G_{\Omega ,m}$ for (\ref{poly}), that is%
\begin{equation}
u\left( x\right) =\left( \mathcal{G}f\right) \left( x\right) :=\int_{\Omega
}G_{\Omega ,m}\left( x,y\right) f\left( y\right) dy  \label{green}
\end{equation}%
solves (\ref{poly}),\ is such that the following estimate holds for some $%
\tilde{c}_{1,\Omega },\tilde{c}_{2,\Omega },\tilde{c}_{3,\Omega }>0$:
\begin{equation}
\tilde{c}_{1,\Omega }H\left( x,y\right) \leq G_{\Omega ,m}\left( x,y\right) +%
\tilde{c}_{2,\Omega }d\left( x\right) ^{m}d\left( y\right) ^{m}\leq \tilde{c}%
_{3,\Omega }H\left( x,y\right)  \label{grsest}
\end{equation}%
with $H:\overline{\Omega }\times \overline{\Omega }\rightarrow \lbrack
0,\infty ]$ defined by%
\begin{equation}
H\left( x,y\right) =\left\{
\begin{array}{cc}
\left\vert x-y\right\vert ^{2m-n}\min \left( 1,\frac{d\left( x\right)
d\left( y\right) }{\left\vert x-y\right\vert ^{2}}\right) ^{m} & \text{for }%
n>2m, \\[3mm]
\log \left( 1+\left( \frac{d\left( x\right) d\left( y\right) }{\left\vert
x-y\right\vert ^{2}}\right) ^{m}\right) & \text{for }n=2m, \\[2mm]
\left( d\left( x\right) d\left( y\right) \right) ^{m-n/2}\min \left( 1,\frac{%
d\left( x\right) d\left( y\right) }{\left\vert x-y\right\vert ^{2}}\right)
^{n/2} & \text{for }n<2m.%
\end{array}%
\right.  \label{Hest}
\end{equation}%
Here $d$ is the distance to the boundary $\partial \Omega $:
\begin{equation*}
d\left( x\right) =d\left( x,\partial \Omega \right) :=\inf \left\{
\left\vert x-x^{\ast }\right\vert ;x^{\ast }\in \partial \Omega \right\} .
\end{equation*}%
The estimate in (\ref{grsest}) allows us to separate the solution operator
in a signed singular part and a smooth bounded part. The singular part will
have the same regularity properties as in (\ref{adnest}) from \cite{ADN1},
but the fixed sign allows us to separate $f^{+}$ and $f^{-}$. \medskip

One may wonder how general such signed regularity estimates may hold for
higher order elliptic boundary value problems. Such estimates are known for
pure powers of the negative laplacian $-\Delta $. Pure powers of second
order elliptic operators with constant coefficients may be allowed and they
may even be perturbed by small lower order terms. See \cite{Pulst}. However,
a recent paper \cite{GRomS} shows examples of higher order elliptic
operators, even with constant coefficients, for which the singularity at $%
x=y $ is sign-changing. Obviously for such a problem there is no estimate
like (\ref{grsest}) possible.

\section{The proof}

The distance function $d$ is at most Lipschitz, even on $C^{\infty }$%
-domains. So as a first step we will replace $d(x)^{m}d(y)^{m}$ in (\ref%
{grsest}) by a smoother function, namely by $w(x)w(y)$ where
\begin{equation}
w:=e_{1}^{m}  \label{w}
\end{equation}
and $e_{1} $ is the solution of%
\begin{equation*}
\left\{
\begin{array}{cc}
-\Delta e_{1}=1 & \text{in }\Omega , \\
e_{1}=0 & \text{on }\partial \Omega .%
\end{array}%
\right.
\end{equation*}

\begin{lemma}
\label{Dlem}This function $w$ from (\ref{w}) inherits the smoothness of the
boundary $\partial \Omega $, in the sense that $\partial \Omega \in
C^{2m,\gamma }$ implies $w\in C^{2m,\gamma }(\overline{\Omega })$. Moreover,
there exist $c_{1},c_{2}>0$ such that%
\begin{equation}
c_{1}~d\left( x\right) ^{m}\leq w\left( x\right) \leq c_{2}~d\left( x\right)
^{m}\text{ for all }x\in \Omega .  \label{wed}
\end{equation}
\end{lemma}

\begin{proof}
By the maximum principle and more precisely a uniform Hopf's boundary point
lemma, which holds for $\partial \Omega \in C^{1,\gamma }$, one finds that\
a constant $C_{\text{H}}>0$ exists with $e_{1}\left( x\right) \geq C_{\text{H%
}} ~d(x)$ for all $x\in \overline{\Omega }$. Since $e_{1}\in C^{1}(\overline{%
\Omega })\cap C_{0}(\overline{\Omega })$, one finds another constant $C_e>0$
such that $e_{1}(x)\leq C_e~d(x)$ for all $x\in \Omega $. By \cite[Theorem
6.19]{GT} $e_{1}\in C^{2m,\gamma }(\overline{\Omega })$ and hence $%
w=e_{1}^{m}\in C^{2m,\gamma }(\overline{\Omega })$ and satisfies (\ref{wed}).
\end{proof}

Since the function $H\left( \cdot ,\cdot \right) $ from (\ref{Hest})
satisfies for some $C_{\text{GRS}}>0$
\begin{equation*}
H\left( x,y\right) \geq C_{\text{GRS}}\left( d(x)d(y)\right) ^{m}\text{ for
all }x,y\in \overline{\Omega }
\end{equation*}%
there exists $\hat{c}_{1},\hat{c}_{2},\hat{c}_{3}>0$ such that the following
variant of (\ref{grsest}) holds:
\begin{equation}
\hat{c}_{1}H\left( x,y\right) \leq G_{\Omega ,m}\left( x,y\right) +\hat{c}%
_{2}\ w(x)\ w(y)\leq \hat{c}_{3}H\left( x,y\right) .  \label{grsestbis}
\end{equation}

We do not directly replace $d$ in $H$ by $e_{1}$, but instead define the
function $H_{\Omega ,m}:\overline{\Omega }\times \overline{\Omega }%
\rightarrow \left[ 0,\infty \right] $ by%
\begin{equation}
H_{\Omega ,m}\left( x,y\right) :=G_{\Omega ,m}\left( x,y\right) +\hat{c}%
_{2}\ w\left( x\right) \,w\left( y\right) .\text{ }  \label{hem}
\end{equation}%
In the next theorem we will state some properties of the operator $\mathcal{H%
}:C(\overline{\Omega })\rightarrow C(\overline{\Omega })$ defined by%
\begin{equation}
\left( \mathcal{H}f\right) (x):=\int_{\Omega }H_{\Omega ,m}\left( x,y\right)
f(y)dy.  \label{hdef}
\end{equation}%
For later use we also set $\mathcal{D}=\mathcal{H}-\mathcal{G}$, i.e.%
\begin{equation}
\left( \mathcal{D}f\right) \left( x\right) :=\hat{c}_{2}\ w\left( x\right)
\int_{\Omega }w\left( y\right) f\left( y\right) dy,  \label{ddef}
\end{equation}%
which is well-defined for $f\in L^{1}\left( \Omega \right) $ and bounded as
operator from $L^{1}\left( \Omega \right) $ to $C^{2m,\gamma }(\overline{%
\Omega })$. Note that $w\in W_{0}^{m,q}(\Omega )$ for any $q\in \left(
1,\infty \right) $ and hence $\mathcal{D}$ can even be extended to $%
W^{-m,p}(\Omega ):=\left( W_{0}^{m,p/\left( p-1\right) }(\Omega )\right)
^{\prime }$.

For $n>2m$ one finds from (\ref{grsestbis}) and (\ref{Hest}) that there is $%
C_{2,\Omega ,m}>0$ such that
\begin{equation}
0\leq H_{\Omega ,m}\left( x,y\right) \leq C_{2,m}\ \left\vert x-y\right\vert
^{2m-n}\text{ for all }\left( x,y\right) \in \overline{\Omega }\times\overline{\Omega },
\label{compRP}
\end{equation}%
Hence from the Hardy-Littlewood-Sobolev Theorem of fractional integration
(see \cite[Theorem 1, p. 119]{Stein}) one finds the first two statements of:

\begin{lemma}
\label{rplem}Let $\Omega $ be a bounded domain and $\partial \Omega \in
C^{2m,\gamma }$ for some $\gamma \in \left( 0,1\right) $, and let $p\in %
\left[ 1,\infty \right) $. Then $\mathcal{H}f$ in (\ref{hem})-(\ref{hdef}) is
well-defined for all $f\in L^{p}\left( \Omega \right) $:

\begin{enumerate}
\item For all $p\in \left[ 1,\infty \right) $ and $f\in L^{p}\left( \Omega
\right) $ the integral in (\ref{hdef}) is absolute convergent for almost
every $x\in \Omega $.

\item Suppose $n>2m$. For all $p\in \left( 1,\frac{n}{2m}\right) $ and $q\in %
\left[ 1,\frac{np}{n-2mp}\right] $ there exist constants $C_{p,q}>0$
independent of $f\in L^{p}\left( \Omega \right) $, such that
\begin{equation}
\left\Vert \mathcal{H}f\right\Vert _{L^{q}\left( \Omega \right) }\leq
C_{p,q}\ \left\Vert f\right\Vert _{L^{p}\left( \Omega \right) }.
\label{Soblem}
\end{equation}

\item For all $p\geq 1$ with $p>\frac{n}{2m}$ there exist constants $C_{p}>0$
independent of $f\in L^{p}\left( \Omega \right) $, such that
\begin{equation}
\left\Vert \mathcal{H}f\right\Vert _{L^{\infty }\left( \Omega \right) }\leq
C_{p}\ \left\Vert f\right\Vert _{L^{p}\left( \Omega \right) }.
\end{equation}
\end{enumerate}
\end{lemma}

\begin{proof}
By (\ref{compRP}) we find for $f\geq 0$ and $f$ extended by $0$ outside of $%
\Omega $, that%
\begin{equation*}
0\leq \mathcal{H}f\leq c_{m}\mathcal{I}_{2m}f
\end{equation*}%
with $\mathcal{I}_{2m}f$ a Riesz potential of $f$ defined in (\ref{rp}). By
\cite[Theorem 1, p. 119]{Stein} the first item holds. Moreover, \cite[%
Theorem 1, p. 119]{Stein} also states that for $p\in \left( 1,\frac{n}{2m}%
\right) $ there exists $C_{\mathrm{HLS},p,n,m}>0$ such that
\begin{equation}
\left\Vert \mathcal{I}_{2m}f\right\Vert _{L^{\frac{np}{n-2mp}}\left( \Omega
\right) }\leq C_{\mathrm{HLS},p,n,m}\ \left\Vert f\right\Vert _{L^{p}\left(
\Omega \right) }.  \label{riepo}
\end{equation}%
With (\ref{riepo}) one finds (\ref{Soblem}) for positive $f$ and $q=\frac{np%
}{n-2mp}$. Since $\Omega $ is bounded, the estimate holds for all
\mbox{$q\in [ 1,\frac{np}{n-2mp}] $}. For general $f$ one splits by $%
f=f^{+}-f^{-}$, uses linearity, $\left\Vert f^{\pm }\right\Vert
_{L^{p}\left( \Omega \right) }\leq \left\Vert f\right\Vert _{L^{p}\left(
\Omega \right) }$ and finds (\ref{Soblem}) with twice the constant for $f$
with fixed sign.

If $n>2m$ and $p>\frac{n}{2m}$, then the third item follows from (\ref%
{compRP}) and the usual estimate by H\"{o}lder's inequality applied to the
Riesz potential:%
\begin{equation*}
\left\Vert \mathcal{H}f\right\Vert _{L^{\infty }\left( \Omega \right) }\leq
c_{m}^{\prime }\sup_{x\in \Omega }\left\Vert \left\vert x-\cdot \right\vert
^{2m-n}\right\Vert _{L^{p/\left( p-1\right) }\left( \Omega \right)
}\left\Vert f\right\Vert _{L^{p}\left( \Omega \right) }
\end{equation*}
For $n=2m$ the logarithmic singularity lies in $L^{q}$ for any $q<\infty $,
which gives the estimate by H\"{o}lder for any $p>1$. For $n<2m$ the kernel
of $\mathcal{H}$ is uniformly bounded, which yields the estimate for $p=1$
and hence for any $p\geq 1$.
\end{proof}

\begin{proposition}

\label{Hthe}Let $\Omega \subset \mathbb{R}^{n}$ be a bounded domain with $%
\partial \Omega \in C^{2m,\gamma }$ for some $\gamma \in \left( 0,1\right) $%
. Let $\mathcal{H}$ be defined by (\ref{hdef},\ref{hem}). Then for any $p\in
\left( 1,\infty \right) $ there exists $C_{m,p}>0$ such that for all $f\in
L^{p}(\Omega )$,\ it holds that $\mathcal{H}f\in W^{2m,p}\left( \Omega
\right) \cap W_{0}^{m,p}(\Omega )$ and%
\begin{equation}
\left\Vert \mathcal{H}f\right\Vert _{W^{2m,p}\left( \Omega \right) }\leq
C_{m,p}\left\Vert f\right\Vert _{L^{p}\left( \Omega \right) }.  \label{Hreg}
\end{equation}
\end{proposition}

\begin{proof}
For $0\leq f\in L^{p}\left( \Omega \right) $ let $f_{\varepsilon }:=\varphi
_{\varepsilon }\ast f\in C(\overline{\Omega })$ denote the usual
mollification, with $\varphi _{\varepsilon }$ the mollifier from Friedrichs
and $f$ extended by $0$ outside of $\Omega $. By Lemma \ref{rplem} and
suitable \thinspace $p,q$, that is $\frac{1}{q}\geq \frac{1}{p}-\frac{2m}{n}$%
, one finds%
\begin{equation}
\left\Vert \mathcal{H}f-\mathcal{H}f_{\varepsilon }\right\Vert
_{L^{q}(\Omega )}\leq C_{p,q}\left\Vert f-f_{\varepsilon }\right\Vert
_{L^{p}(\Omega )}.  \label{heps}
\end{equation}%
Also $\mathcal{D}$ is well-defined for $f\in L^{1}\left( \Omega \right) $,
which contains $L^{p}(\Omega )$, and since $w\in C^{2m,\gamma }(\overline{%
\Omega })$ holds by Lemma \ref{Dlem}, one even finds for some $c_{m}>0$ that
\begin{equation}
\left\Vert \mathcal{D}f\right\Vert _{C^{2m,\gamma }(\overline{\Omega })}\leq
c_{m}~\left\Vert f\right\Vert _{L^{1}\left( \Omega \right) },  \label{Dest}
\end{equation}%
and hence also that
\begin{equation}
\left\Vert \mathcal{D}f-\mathcal{D}f_{\varepsilon }\right\Vert
_{L^{q}(\Omega )}\leq c_{m,p,q}^{\prime }~\left\Vert f-f_{\varepsilon
}\right\Vert _{L^{p}(\Omega )}.  \label{deps}
\end{equation}%
Replacing $f$ on the right-hand side of (\ref{poly}) by $f_{\varepsilon }$,
we find as solution
\begin{equation}
\mathcal{G}f_{\varepsilon }=\mathcal{H}f_{\varepsilon }-\mathcal{D}%
f_{\varepsilon }.  \label{eee}
\end{equation}%
Letting $u\in W^{2m,p}(\Omega )\cap W_{0}^{m,p}(\Omega )$ denote the
solution of (\ref{poly}) for $f$ on the right-hand side, we find by \cite%
{ADN1} that
\begin{equation*}
\left\Vert u-\mathcal{G}f_{\varepsilon }\right\Vert _{W^{2m,p}\left( \Omega
\right) }\leq C_{\mathrm{ADN},2m,p}\left\Vert f-f_{\varepsilon }\right\Vert
_{L^{p}\left( \Omega \right) }\rightarrow 0\text{ for }\varepsilon
\downarrow 0.
\end{equation*}%
From (\ref{heps}), (\ref{deps}) and (\ref{eee}) one finds%
\begin{equation*}
\left\Vert \left( \mathcal{H}-\mathcal{D}\right) f-\mathcal{G}f_{\varepsilon
}\right\Vert _{L^{q}(\Omega )}\leq c_{m,p,q}~\left\Vert f-f_{\varepsilon
}\right\Vert _{L^{p}(\Omega )}\rightarrow 0\text{ for }\varepsilon
\downarrow 0.
\end{equation*}
Hence $\left( \mathcal{H}-\mathcal{D}\right) f=u\in W^{2m,p}(\Omega )\cap
W_{0}^{m,p}(\Omega )$. With $\mathcal{D}f\in C^{2m,\gamma }(\overline{\Omega
})\cap C_{0}^{m-1}(\overline{\Omega })$ one also finds
\begin{equation*}
\mathcal{H}f\in W^{2m,p}(\Omega )\cap W_{0}^{m,p}(\Omega ).
\end{equation*}
Moreover, there exists $C_{4,m,p}>0$ such that
\begin{gather*}
\left\Vert \mathcal{H}f\right\Vert _{W^{2m,p}\left( \Omega \right)
}=\left\Vert u+\mathcal{D}f\right\Vert _{W^{2m,p}\left( \Omega \right) }\leq
\left\Vert u\right\Vert _{W^{2m,p}\left( \Omega \right) }+\left\Vert
\mathcal{D}f\right\Vert _{W^{2m,p}\left( \Omega \right) } \\
\leq C_{\mathrm{ADN},2m,p}\left\Vert f\right\Vert _{L^{p}\left( \Omega
\right) }+c_{m}^{\prime }\ \left\Vert f\right\Vert _{L^{1}\left( \Omega
\right) }\leq C_{4,m,p}\left\Vert f\right\Vert _{L^{p}\left( \Omega \right)
}.
\end{gather*}
\end{proof}

It remains to combine these results to the statement of Theorem \ref%
{plusminusregthm}. \medskip

\begin{proof}[Proof of Theorem \protect\ref{plusminusregthm}]
Since $\mathcal{H}$ and $\mathcal{D}$ preserve the sign, we may consider
separately the solutions $u_{+}$ and $u_{-}$ of%
\begin{equation*}
\left\{
\begin{array}{cc}
\left( -\Delta \right) ^{m}u_{\pm }=f^{\pm } & \text{in }\Omega , \\
D^{\alpha }u_{\pm }=0\text{ for }\left\vert \alpha \right\vert <m & \text{on
}\partial \Omega .%
\end{array}%
\right.
\end{equation*}%
Here $u_{\pm }$ does not have a sign but just denotes the solution parts
depending on the signed splitting of the right-hand side $f^{\pm }$.
One finds%
\begin{gather*}
u_{+}\left( x\right) =\left( \mathcal{H}f^{+}\right) \left( x\right) -\left(
\mathcal{D}f^{+}\right) \left( x\right) , \\
u_{-}\left( x\right) =\left( \mathcal{H}f^{-}\right) \left( x\right) -\left(
\mathcal{D}f^{-}\right) \left( x\right) .
\end{gather*}%
So
\begin{equation*}
u\left( x\right) =\left( \mathcal{H}f^{+}\right) \left( x\right) -\left(
\mathcal{D}f^{+}\right) \left( x\right) -\left( \mathcal{H}f^{-}\right)
\left( x\right) +\left( \mathcal{D}f^{-}\right) \left( x\right) .
\end{equation*}%
We split this expression into two parts:%
\begin{gather*}
u^{\oplus }\left( x\right) =\left( \mathcal{H}f^{+}\right) \left( x\right)
+\left( \mathcal{D}f^{-}\right) \left( x\right), \\
u^{\ominus }\left( x\right) =\left( \mathcal{H}f^{-}\right) \left( x\right)
+\left( \mathcal{D}f^{+}\right) \left( x\right),
\end{gather*}%
and $u\left( x\right) =u^{\oplus }\left( x\right) -u^{\ominus }\left(
x\right) $ with both parts $u^{\oplus },u^{\ominus }$ being nonnegative. For
the $\mathcal{H}$-part of $u^{\oplus },u^{\ominus }$ we use the results of
Proposition \ref{Hthe}. The estimate in (\ref{Dest}) takes care of the $\mathcal{%
D}$-part.
\end{proof}

\end{document}